\documentclass[a4paper,twoside]{article}

\usepackage{amsthm,amsmath,amssymb}
\usepackage[bookmarks=true,bookmarksopen=true]{hyperref}

\usepackage{xspace}

\numberwithin{equation}{section}
\theoremstyle{plain}
\newtheorem{thrm}{Theorem}[section]
\newtheorem{lmm}[thrm]{Lemma}
\newtheorem{crllr}[thrm]{Corollary}

\theoremstyle{definition}
\newtheorem{dfntn}[thrm]{Definition}

\theoremstyle{remark}

\newcommand{\Pb}{\mbox{\rm (P)}\xspace}
\newcommand{\uad}{U_{\rm ad}}
\renewcommand{\div}{\operatorname{div}}

\title{Analysis of control problems of nonmontone semilinear elliptic equations\thanks{The first two authors were partially supported by Spanish Ministerio de Econom\'{\i}a y Competitividad under research project MTM2017-83185-P.}}

\author{Eduardo Casas\thanks{Departmento de Matem\'{a}tica Aplicada y Ciencias de la Computaci\'{o}n, E.T.S.I. Industriales y de Telecomunicaci\'on, Universidad de Cantabria, 39005 Santander, Spain, {\tt eduardo.casas@unican.es}.}
\and
Mariano Mateos\thanks{Departamento de Matem\'{a}ticas, Campus de Gij\'on, Universidad de Oviedo, 33203, Gij\'on, Spain, {\tt mmateos@uniovi.es}.}
\and
Arnd R\"osch\thanks{Fakult\"at f\"ur Mathematik, Universt\"at Duisburg-Essen, D-45127 Essen, Germany,
{\tt arnd.roesch@uni-due.de}}
}

\begin{document}

\maketitle

\begin{abstract}
In this paper we study optimal control problems governed by a semilinear elliptic equation. The equation is nonmonotone due to the presence of a convection term, despite the monotonocity of the nonlinear term. The resulting operator is neither monotone nor coervive. However, by using conveniently a comparison principle we prove existence and uniqueness of solution for the state equation. In addition, we prove some regularity of the solution and differentiability of the relation control-to-state. This allows us to derive first and second order conditions for local optimality.
\end{abstract}

\begin{quote}
\end{quote}
\begin{quote}
\end{quote}

\section{Introduction}
\label{S1}
In this paper, we consider an optimal control problem associated with the following elliptic semilinear equation
\begin{equation}
\left\{\begin{array}{l} Ay + b(x)\cdot\nabla y + f(x,y) = u \text{ in } \Omega,\\ y = 0\text{ on } \Gamma,\end{array}\right.
\label{E1.1}
\end{equation}
where $A$ is an elliptic operator, $b:\Omega \longrightarrow \mathbb{R}^n$ is a given function, $f:\Omega \times \mathbb{R} \longrightarrow \mathbb{R}$ is nondecreasing monotone in the second variable, $u \in L^2(\Omega)$, $\Omega$ is a domain in $\mathbb{R}^n$, $n= 2$ or 3, and $\Gamma$ is the boundary of $\Omega$. The precise assumptions on these data will be given in the next section. Due to the convection term induced by $b$, the linear part of the above operator is nonmonote. We emphasize that here we neither assume that $\div b = 0$ nor $b$ is small. Consequently, the bilinear form associated with the linear part of the operator is not necessarily  coercive. This introduces some important difficulties in the analysis of the equation. A thorough study is needed to prove existence and  uniqueness of a solution of the equation \eqref{E1.1} for every $u$. This study makes a strong use of a comparison principle.

In many publications $\div b=0$ is assumed.
This property is satisfied in several applications, for instance if the quantity $b$ represents a velocity field of an incompressible Navier-Stokes flow. If the flow is compressible, the assumption $\div b=0$ cannot be justified.
Some examples of applications where the divergence of the convection term need not be zero can be found in the introductory chapter of the book \cite{Banks-Kunisch1989}.

 In many books devoted to partial differential equations, the convection term appears and it is not assumed to have zero divergence. Let us mention the classical books \cite{Evans1998} or \cite{Lad-Ural-elliptic1968}; see also \cite{Stampacchia65}. However, only a few references treat the topic of existence and uniqueness of solution for linear elliptic equations with convection term such that $\div b \ne 0$ and $b$ is not small. When the nonlinear term $f$ is not present in the state equation, the reader is referred to the early reference \cite{Trudinger-1973} for the existence and uniqueness of a solution; see also \cite[Theorem 8.3]{Gilbarg-Trudinger83} or the recent reference \cite{Boccardo2019COCV}. In \cite{Boccardo2019BUMI}, the case of a semilinear equation in dimension $n \ge 3$ with the non-linearity $y|y|^{\lambda -1}$, $\lambda > \frac{n}{n - 2}$, is studied. As far as we know, the most general and complete results for the analysis of equation \eqref{E1.1} are the ones presented below in Section \ref{S2}.

 The case of nonmonotone quasilinear elliptic equations was considered in \cite{Casas-Troltzsch2009} and \cite{Krizek-Liu2003}. However, in the last two papers, the operator was coercive. The equation considered in this paper does not fit in the problems studied in the mentioned references.

Associated with the state equation \eqref{E1.1} we consider the following control problem:
\[
\Pb\quad \min_{u \in \uad} J(u) := \int_\Omega L(x,y_u(x))\, dx + \frac{\nu}{2}\int_\Omega u^2(x)\, dx
\]
where $y_u$ is the solution of \eqref{E1.1} associated with $u$, $L:\Omega \times \mathbb{R} \longrightarrow \mathbb{R}$ is a given function, $\nu > 0$, and
\[
\uad = \{u \in L^2(\Omega) : \alpha \le u(x) \le \beta \text{ for a.a. } x \in \Omega\}
\]
with $-\infty \le \alpha < \beta \le +\infty$. A precise analysis of the state equation allows us to prove the existence of a solution for \Pb as well as to get the first and second order optimality conditions.

Typical examples of nonlinearities in the state equations are $f(x,y) = a_0(x)|y|^ry$ with $r > 0$ or $f(x,y) = a_0(x)\exp(y)$, where $a_0$ is assumed to be nonnegative and bounded. The assumption $r > 1$ is needed to prove the existence of a second derivative of $J$. Concerning the functional $J$, the usual tracking cost functional falls into this framework by setting $L(x,y) = \frac{1}{2}(y - y_d(x))^2$ for some fixed function $y_d \in L^2(\Omega)$.

To our best knowledge, this is the first time that a control problem governed by a nonmonotone and noncoercive equation of the kind described here has been considered. The methods to study the control problem are technically more involved than those used for problems governed by coercive equations, as it can be seen not only in the study of the state equation, but also in the proofs of some results such as Lemma \ref{L3.5} or Theorem \ref{T3.8}.

The paper is organized as follows. In section \ref{S2}, the state equation is analyzed. We address the issues of existence, uniqueness and regularity results of the solution for both the linear and semilinear cases. Differentiability of the relation control-to-state is also established. Finally, the existence of solution for \Pb as well as first and second order optimality conditions are proved in Section \ref{S3}. Based on the results established in this paper, the numerical analysis for \Pb will be carried out in a forthcoming paper.

\section{Analysis of the state equation}
\label{S2}
\setcounter{equation}{0}

In this section we study the equation \eqref{E1.1} proving some results that will be used in the analysis of the control problem \Pb. Before studying \eqref{E1.1}, we analyze a linear equation involving the convection term. The section is divided into two subsections. The first one is devoted to the linear equation and the second to the study of \eqref{E1.1}

\subsection{Study of the linear operator}
\label{S2.1}
The following assumption is needed for this analysis.

\begin{quotation}
\textit{Assumption 1.} $\Omega$ is an open domain in $\mathbb{R}^n$, $n = 2$ or 3, with a Lipschitz boundary $\Gamma$. $A$ is the operator given by
\[
Ay= - \sum_{i,j = 1}^n\partial_{x_j}(a_{ij}(x)\partial_{x_i}y) \ \text{ with }\ a_{ij} \in L^\infty(\Omega),
\]
and satisfying the following ellipticity condition:
\[
\exists \Lambda > 0 \text{ such that } \sum_{i,j = 1}^na_{ij}(x)\xi_i\xi_j \ge \Lambda|\xi|^2\ \ \forall \xi \in \mathbb{R}^n \text{ and for a.a. } x \in \Omega.
\]
The function  $b:\Omega\to\mathbb{R}^n$ satisfies  $b\in L^p(\Omega)^n$ with $p\geq 3$ if $n=3$ and $p > 2$ if $n=2$. For the function $a_0:\Omega\to\mathbb{R}$ it is assumed that $a_0 \in L^q(\Omega)$ with $q > 1$ if $n = 2$ and $q \ge \frac{3}{2}$ if $n = 3$.

\hspace{-\parindent}Unless stated otherwise, in the rest of the paper $\bar p$ will denote a number such that $\bar p > n/2$.
\end{quotation}

\medskip

Notice that with this choice, $L^{\bar p}(\Omega)\subset W^{-1,r}(\Omega)\subset H^{-1}(\Omega)$ for some $r>n$.

Along this paper we will take
\[
\|y\|_{H_0^1(\Omega)} = \left(\int_\Omega|\nabla y(x)|^2\, dx\right)^{\frac{1}{2}}.
\]
From the Poincar\'e inequality and the Sobolev embedding theorem, we know that there exist two constants $C_\Omega$ and $K_\Omega$ such that
\begin{equation}
\label{E2.1}\|y\|_{L^2(\Omega)} \le C_\Omega\|y\|_{H_0^1(\Omega)}\quad\mbox{ and }\quad \|y\|_{L^6(\Omega)} \le K_\Omega\|y\|_{H_0^1(\Omega)}\quad \forall y \in H_0^1(\Omega).
\end{equation}
As a consequence, we have that $\|y\|_{H^{-1}(\Omega)} \le C_\Omega\|y\|_{L^2(\Omega)}$ for all $y \in L^2(\Omega)$.

 Let us consider the elliptic operator
\begin{equation}
\mathcal{A}y = Ay + b(x)\cdot\nabla y + a_0(x) y \ \text{ with } a_0 \ge 0.
\label{E2.2}
\end{equation}

We first prove continuity of this operator and G\r{a}rding's inequality.
\begin{lmm}\label{L2.1}
Under Assumption 1 we have that
$\mathcal{A}\in \mathcal{L}(H^1_0(\Omega),H^{-1}(\Omega))$ and there exists a constant $C_{\Lambda,b}$ such that
\begin{equation}
\langle \mathcal{A} z,z\rangle_{H^{-1}(\Omega),H^1_0(\Omega)}  \ge \frac{\Lambda}{4}\|z\|^2_{H_0^1(\Omega)} - C_{\Lambda,b}\|z\|^2_{L^2(\Omega)}\quad \forall z \in H_0^1(\Omega).
\label{E2.3}
\end{equation}
\end{lmm}

\begin{proof}
Let us show that $\mathcal{A}$ is a linear continuous operator. We will prove the result for dimension $n = 3$ and we can argue in a similar way for dimension $n = 2$. It is obvious that $A:H_0^1(\Omega) \longrightarrow H^{-1}(\Omega)$ is a continuous linear mapping due to the fact that $a_{ij} \in L^\infty(\Omega)$. Moreover, from \eqref{E2.1} and H\"older inequality we infer for every $z \in H_0^1(\Omega)$
\begin{align*}
&\|b\cdot\nabla z\|_{L^{\frac{6}{5}}(\Omega)} \le \|b\|_{L^3(\Omega)^3}\|\nabla z\|_{L^2(\Omega)^3} = \|b\|_{L^3(\Omega)^3}\|z\|_{H_0^1(\Omega)},\\
&\|a_0 z\|_{L^{\frac{6}{5}}(\Omega)} \le \|a_0\|_{L^{\frac{3}{2}}(\Omega)}\|z\|_{L^6(\Omega)} \le K_\Omega\|a_0\|_{L^{\frac{3}{2}}(\Omega)}\|z\|_{H_0^1(\Omega)}.
\end{align*}
Hence, we have that $\mathcal{A}$ is a well-posed linear and continuous operator.

Let us prove \eqref{E2.3}.  In this case, the proofs for $n=3$ and $n=2$ are slightly different. We start with $n=2$. Using that $a_0\geq 0$ and Young and H\"older inequalities we get
\begin{align*}
\langle \mathcal{A} z,z\rangle_{H^{-1}(\Omega),H^1_0(\Omega)} \ge& \Lambda\|\nabla z\|^2_{L^2(\Omega)^n} - \|\nabla z\|_{L^2(\Omega)^n}\|bz\|_{L^2(\Omega)^{n}} \ge \frac{\Lambda}{2}\|\nabla z\|^2_{L^2(\Omega)^n} - \frac{1}{2\Lambda}\|bz\|^2_{L^2(\Omega)^{n}}\\
 \ge &\frac{\Lambda}{2}\|\nabla z\|^2_{L^2(\Omega)^n} - \frac{1}{2\Lambda}\|b\|^2_{L^{p}(\Omega)^n}\|z\|^2_{L^{\frac{2 p}{p - 2}}(\Omega)}.
\end{align*}
Observe that the assumption $p > 2$ implies that $2 < \frac{2 p}{p - 2} < \infty$ if $n = 2$. Now, we apply Lions' Lemma, \cite[Chapter 2, Lemma 6.1]{Necas67}, to the spaces $H_0^1(\Omega) \subset L^{\frac{2 p}{ p - 2}}(\Omega) \subset L^2(\Omega)$ to deduce the existence of a constant $C_0$ depending of $\Lambda$ and $\|b\|_{L^{p}(\Omega)^n}$ such that
\[
\|z\|_{L^{\frac{2p}{p - 2}}(\Omega)} \le \frac{\Lambda}{2\|b\|_{L^{p}(\Omega)^n}}\|\nabla z\|_{L^2(\Omega)^n} + C_0\|z\|_{L^2(\Omega)}.
\]
From the last two inequalities we conclude \eqref{E2.3}
with
\[
C_{\Lambda,b} = \frac{C_0^2\|b\|^2_{L^{p}(\Omega)^n}}{\Lambda}.
\]
For $n=3$ we proceed as follows. From \cite[Lemma 3.1]{Stampacchia65}, we know that, for any $\varepsilon >0 $ there exists a constant $K_{\varepsilon,b}>0$ depending on $b$ and $\varepsilon$ such that $b = b'+b''$, with $\|b'\|_{L^\infty(\Omega)^{n}} < K_{\varepsilon,b} $ and $\|b''\|_{L^3(\Omega)^{n}} < \varepsilon$. Taking $\varepsilon = \Lambda/(4 K_\Omega)$, $K_\Omega$ satisfying \eqref{E2.1}, and using that $a_0\geq 0$, H\"older and Young inequalities and \eqref{E2.1}, we obtain
\begin{align*}
\langle \mathcal{A} z,z\rangle_{H^{-1}(\Omega),H^1_0(\Omega)} \ge& \Lambda\|\nabla z\|^2_{L^2(\Omega)^n} - \int_\Omega(b'+b'')\cdot\nabla z z dx \\
    \ge &  \Lambda|\nabla z\|^2_{L^2(\Omega)^n} - \|b'\|_{L^\infty(\Omega)^{n}}\|\nabla z\|_{L^2(\Omega)^{n}}\|z\|_{L^2(\Omega)} -
    \|b''\|_{L^3(\Omega)^{n}}\|\nabla z\|_{L^2(\Omega)^{n}}\|z\|_{L^6(\Omega)}  \\
 \ge &\Lambda\|\nabla z\|^2_{L^2(\Omega)^n} - \frac{\Lambda}{2}\|\nabla z\|_{L^2(\Omega)^{n}}^2 - \frac{K_{\varepsilon,b}^2}{2\Lambda}\|z\|_{L^2(\Omega)}^2
 - \frac{\Lambda}{4}\|\nabla z\|_{L^2(\Omega)^{n}}^2
\end{align*}
and \eqref{E2.3} follows with a constant $C_{\Lambda,b} = K_{\varepsilon,b}^2/(2\Lambda)$.
\end{proof}

\begin{thrm}
Under Assumption 1, the linear operator $\mathcal{A}:H_0^1(\Omega) \longrightarrow H^{-1}(\Omega)$ is an isomorphism.
\label{T2.2}
\end{thrm}
This theorem can be deduced from the results in \cite{Trudinger-1973}. We include a direct proof for the convenience of the reader.
\begin{proof}
From Lemma \ref{L2.1} we know that $\mathcal{A}$ is a well-posed linear and continuous operator. Let us divide the proof into three steps.

{\em Step 1. -- $\mathcal{A}$ is injective.} We make this proof for $n = 3$. The case $n = 2$ follows along the same lines with minor changes. To prove that the kernel of $\mathcal{A}$ is reduced to 0 we adapt the proof of Theorem 8.1 in \cite{Gilbarg-Trudinger83}. Let $y \in H_0^1(\Omega)$ satisfy that $\mathcal{A}y = 0$. We prove that $y \le 0$ in $\Omega$, the contrary inequality follows by arguing on $-y$. We argue by contradiction and we suppose that this is false. Then, we take $0 < \rho < \text{\rm ess sup}_{x \in \Omega}y(x)\leq +\infty$
and we define $y_\rho(x) = (y(x) - \rho)^+$. Obviously we have that $y_\rho \in H_0^1(\Omega)$. We denote $\Omega_\rho = \{x \in \Omega : \nabla y_\rho(x) \neq 0\}$, then
\[
\nabla y_\rho (x) = \left\{\begin{array}{cl} \nabla y(x) & \text{if } y(x) > \rho,\\0 & \text{otherwise,}\end{array}\right. \text{ and } y_\rho (x) = 0  \text{ if }  y(x) \le \rho.
\]
Using these facts and our assumptions on $b$ and $a_0$ we get
\begin{align*}
&0 = \int_\Omega\left(\sum_{i, j = 1}^n a_{ij}(x)\partial_{x_i}y\partial_{x_j}y_\rho+ [b(x)\cdot\nabla y] y_\rho + a_0(x) y y_\rho\right)\, dx\\
& \ge \int_{\Omega_\rho}\left(\sum_{i, j = 1}^n a_{ij}(x)\partial_{x_i}y_\rho\partial_{x_j}y_\rho + [b(x)\cdot\nabla y_\rho] y_\rho\right)\, dx\\
& \ge \Lambda\|\nabla y_\rho\|^2_{L^2(\Omega_\rho)^{n}} - \|b\|_{L^3(\Omega_\rho)^{n}}\|\nabla y_\rho\|_{L^2(\Omega_\rho)^{n}}\|y_\rho\|_{L^6(\Omega_\rho)}.
\end{align*}
From here and \eqref{E2.1} we infer
\begin{align*}
&\|y_\rho\|_{L^6(\Omega_\rho)} \le \|y_\rho\|_{L^6(\Omega)} \le K_\Omega\|\nabla y_\rho\|_{L^2(\Omega)^{n}}\\
& = K_\Omega\|\nabla y_\rho\|_{L^2(\Omega_\rho)^{n}} \le \frac{1}{\Lambda}K_\Omega\|b\|_{L^3(\Omega_\rho)^{n}}\|y_\rho\|_{L^6(\Omega_\rho)}.
\end{align*}
Hence, we have
\[
\|b\|_{L^3(\Omega_\rho)^{n}} \ge \frac{\Lambda}{K_\Omega} > 0,
\]
which contradicts the fact that $|\Omega_\rho| \to 0$ if $\rho \to \text{\rm ess sup}_{x \in \Omega}y(x)$. Indeed, if the set of points $E = \{x \in \Omega :  y(x) = \text{\rm ess sup}_{x \in \Omega}y(x)\}$ has zero Lebesgue measure, then it is obvious that $|\Omega_\rho| \to 0$ if $\rho \to \text{\rm ess sup}_{x \in \Omega}y(x)$. In the case that $|E| > 0$, then we have that $\nabla y(x) = 0$ a.e.~in $E$ \cite[Lemma 7.7]{Gilbarg-Trudinger83}, and consequently $\nabla y_\rho(x) = \nabla y(x) = 0$ a.e.~in $E$ as well. Hence, $|\Omega_\rho| \to 0$ holds in any case.

We notice that the same procedure can be used to prove the injectivity of $\mathcal{A}^*:H^1_0(\Omega)\to H^{-1}(\Omega)$, given by
\[\mathcal{A}^*\varphi = A^*\varphi -\div (\varphi b(x) ) + a_0\varphi.\]
% REMOVED on June 8, 2020
%To do this, we take as before $\varphi\in H^1_0(\Omega)$ such that $\mathcal{A}^*\varphi =0$, $0<\rho<\text{\rm ess sup}_{x \in \Omega}\varphi(x)$, and $\varphi_\rho(x) = (\varphi(x)-\rho)^+$.
%Using integration by parts we have
%\begin{align*}
%  0 = & \int_\Omega \left(\sum_{i, j = 1}^n a_{ij}(x)\partial_{x_i}\varphi\partial_{x_j}\varphi_\rho\, dx -\div (\varphi b(x)  ) \varphi_\rho + a_0(x) \varphi \varphi_\rho\right)\, dx  \\
%   =  & \int_\Omega \left(\sum_{i, j = 1}^n a_{ij}(x)\partial_{x_i}\varphi\partial_{x_j}\varphi_\rho\, dx +   \varphi b(x)\cdot \nabla \varphi_\rho + a_0(x) \varphi \varphi_\rho\right) dx
%\end{align*}
%and we can argue exactly as we did for $\mathcal{A}$.
% ADDED on June 8, 2020
To do this, we take $\varphi\in H^1_0(\Omega)$ such that $\mathcal{A}^*\varphi =0$, and define for all $\varepsilon \geq 0$, $\Omega^{\varepsilon}=\{x\in\Omega: |\varphi(x)|>\varepsilon\}$ and $\varphi^\varepsilon(x)=\mathrm{proj}_{[-\varepsilon,\varepsilon]}(\varphi(x))$. Using integration by parts, that $a_0(x)\varphi(x)\varphi^\varepsilon(x)\geq 0$ and the fact that $\nabla\varphi^\varepsilon=0$ in $\Omega^\varepsilon$, we have
\begin{align*}
  0 = & \int_\Omega \left(\sum_{i, j = 1}^n a_{ij}(x)\partial_{x_i}\varphi\partial_{x_j}\varphi^\varepsilon -\div (\varphi b(x)  ) \varphi^\varepsilon + a_0(x) \varphi \varphi^\varepsilon\right)\, dx  \\
   =  & \int_\Omega \left(\sum_{i, j = 1}^n a_{ij}(x)\partial_{x_i}\varphi\partial_{x_j}\varphi^\varepsilon +   \varphi b(x)\cdot \nabla \varphi^\varepsilon + a_0(x) \varphi \varphi^\varepsilon\right) dx\\
   \geq & \Lambda \|\nabla \varphi^\varepsilon\|^2_{L^2(\Omega)^n} -
   \|b\|_{L^3(\Omega^0\setminus \Omega^\varepsilon)^n}
   \|\nabla \varphi^\varepsilon\|_{L^2(\Omega)^n}
   \|\varphi^\varepsilon\|_{L^6(\Omega^0\setminus \Omega^\varepsilon)}.
\end{align*}
From here we infer that
\[\|\nabla \varphi^\varepsilon\|_{L^2(\Omega)^n}\leq \frac{1}{\Lambda}\|b\|_{L^3(\Omega^0\setminus \Omega^\varepsilon)^n}
      \|\varphi^\varepsilon\|_{L^6(\Omega^0\setminus \Omega^\varepsilon)}\leq
      \frac{1}{\Lambda}\varepsilon |\Omega^0\setminus \Omega^\varepsilon|^{\frac{1}{6}}\|b\|_{L^3(\Omega^0\setminus \Omega^\varepsilon)^n}.\]
Using this and \eqref{E2.1} we get
\begin{align*}
  |\Omega^\varepsilon| =  & \frac{1}{\varepsilon^2}\int_{\Omega^\varepsilon}\varphi^\varepsilon(x)^2 dx  \leq
  \frac{1}{\varepsilon^2}\int_{\Omega}\varphi^\varepsilon(x)^2 dx \leq
  \frac{1}{\varepsilon^2}C_\Omega^2\|\nabla\varphi^\varepsilon\|^2_{L^2(\Omega)^n}\leq
        \frac{C_\Omega^2}{\Lambda^2} |\Omega^0\setminus \Omega^\varepsilon|^{\frac{1}{3}}\|b\|^2_{L^3(\Omega^0\setminus \Omega^\varepsilon)^n},
\end{align*}
and $|\Omega^0|=\lim_{\varepsilon\to 0}|\Omega^\varepsilon| = 0$ and, hence, $\varphi=0$.
% END OF ADDED TEXT

{\em Step 2. -- The range of $\mathcal{A}$ is dense and closed.} To see that it is dense, we argue by contradiction: suppose it is not dense. Then, there exists $z\in H^{1}_0(\Omega)$ with $z \neq 0$ such that $\langle \mathcal{A} y,z\rangle_{H^{-1}(\Omega),H^1_0(\Omega)} =0$ for all $y\in H^1_0(\Omega)$. By duality, this implies  $\langle \mathcal{A}^*z,y\rangle_{H^{-1}(\Omega),H^1_0(\Omega)} =0$ for all $y\in H^1_0(\Omega)$, and hence $ \mathcal{A}^*z =0$. Since $\mathcal{A}^*$ is injective, we obtain that $z=0$, which is a contradiction.

Let us check that the range of $\mathcal{A}$ is closed. Let $\{f_k\}_{k = 1}^\infty$ be a sequence in the range of $\mathcal{A}$ such that $f_k\to f$ in $H^{-1}(\Omega)$. Let  $y_k\in H^1_0(\Omega)$ be such that $\mathcal{A} y_k=f_k$ for every $k \ge 1$. We are going to prove that $\{y_k\}_{k = 1}^\infty$ converges weakly in $H_0^1(\Omega)$ to some element $y \in H_0^1(\Omega)$ satisfying $\mathcal{A}y = f$.

First, let us prove that $y_k$ is bounded in $L^2(\Omega)$.  We argue by contradiction. Suppose it is not. Then, for a subsequence denoted in the same form, we have that $\|y_k\|_{L^2(\Omega)}\to +\infty$. Define $\hat y_k = y_k/\|y_k\|_{L^2(\Omega)}$, and $\hat f_k = f_k/\|y_k\|_{L^2(\Omega)}$. We have that $\mathcal{A}\hat y_k = \hat f_k$. Using again G\r{a}rding's inequality and the fact that $\|\hat y_k\|_{L^2(\Omega)} = 1$, we obtain
\begin{align*}
  \frac{\Lambda}{4} \|\hat y_k\|_{H^1_0(\Omega)}^2 \leq & \langle\mathcal{A}\hat y_k,\hat y_k\rangle_{H^{-1}(\Omega),H^1_0(\Omega)} + C_{\Lambda,b}\|\hat y_k\|_{L^2(\Omega)}^2 =
                                                      \langle \hat f_k,\hat y_k\rangle_{H^{-1}(\Omega),H^1_0(\Omega)} + C_{\Lambda,b}\|\hat y_k\|_{L^2(\Omega)}^2 \\
  \leq  & \|\hat f_k\|_{H^{-1}(\Omega)} \|\hat y_k\|_{H^1_0(\Omega)} + C_{\Lambda,b}\leq \frac{2}{\Lambda}\|\hat f_k\|^2_{H^{-1}(\Omega)}+\frac{\Lambda}{8} \|\hat y_k\|^2_{H^1_0(\Omega)} + C_{\Lambda,b},
\end{align*}
and therefore
\[\frac{\Lambda}{8} \|\hat y_k\|_{H^1_0(\Omega)}^2 \leq \frac{2}{\Lambda}\|\hat f_k\|^2_{H^{-1}(\Omega)}+ C_{\Lambda,b}.\]
Since $\{f_k\}_{k = 1}^\infty$ is bounded in $H^{-1}(\Omega)$, we infer from the above inequality that $\{\hat y_k\}_{k = 1}^\infty$ is bounded in $H^1_0(\Omega)$. Hence, there exists $\hat y\in H^1_0(\Omega)$ such that (for a new subsequence, again denoted in the same way) $\hat y_k\rightharpoonup \hat y$ in $H^1_0(\Omega)$, and by Rellich's theorem $\hat y_k\to \hat y$ in $L^2(\Omega)$. In particular, this implies that $\|\hat y\|_{L^2(\Omega)} =1$.
On the other hand, using that $\|\hat f_k\|_{H^{-1}(\Omega)} \to 0$ and $\mathcal{A}\hat y_k = \hat f_k$, we can pass to the limit in this equation and deduce that $\mathcal{A} \hat y = 0$. Hence $\hat y=0$ and we have a contradiction. Thus,  $\{y_k\}_{k = 1}^\infty$ is bounded in $L^2(\Omega)$.

Using again G\r{a}rding's inequality for the equation $\mathcal{A}y_k =f_k$, we obtain as above that
\[\frac{\Lambda}{8} \|y_k\|_{H^1_0(\Omega)}^2 \leq \frac{2}{\Lambda}\|f_k\|^2_{H^{-1}(\Omega)}+ C_{\Lambda,b}\|y_k\|_{L^2(\Omega)}^2,\]
and therefore, $\{y_k\}_{k=1}^\infty$ is bounded in $H^1_0(\Omega)$. So there exists $y\in H^1_0(\Omega)$ and a subsequence, denoted in the same way, such that $y_k\rightharpoonup y$ in $H^1_0(\Omega)$. Taking the limit in the equation $\mathcal{A}  y_k = f_k$, we have that $\mathcal{A}  y = f$. Hence, $f$ belongs to the range of $\mathcal{A}$ and this subspace is closed in $H_0^1(\Omega)$.

\end{proof}

The following corollary is a straightforward application of Theorem \ref{T2.2} and the H\"older regularity result \cite[Th\'eor\`eme 7.3] {Stampacchia65}; see also Theorem 14.1  in \cite{Lad-Ural-elliptic1968} and the remark after it.
\begin{crllr}
Suppose Assumption 1 holds. Then, for every $u\in L^{\bar p}(\Omega)$ there exists a unique function $y \in H_0^1(\Omega)\cap C^{0,\mu}(\bar\Omega)$, for some $\mu \in (0,1)$  independent of $u$, satisfying $\mathcal{A}y = u$. Moreover, there exists a constant $C_{\mathcal{A},\mu}$ such that
\begin{equation}
\|y\|_{C^{0,\mu}(\bar\Omega)} \le C_{\mathcal{A},\mu}\|u\|_{L^{\bar p}(\Omega)}\quad \forall u \in L^{\bar p}(\Omega).
\label{E2.4}
\end{equation}
\label{C2.3}
\end{crllr}

The adjoint operator also enjoys these properties.

\begin{crllr}
Under Assumption 1, the adjoint operator $\mathcal{A}^*:H_0^1(\Omega) \longrightarrow H^{-1}(\Omega)$ given by
\begin{equation}
\mathcal{A}^*\varphi = A^*\varphi - \div[b(x)\varphi] + a_0(x)\varphi
\label{E2.5}
\end{equation}
is an isomorphism. Moreover, for every $f\in  L^{\bar p}(\Omega)$, there exists a unique $\varphi \in H_0^1(\Omega)$ satisfying $\mathcal{A}^*\varphi = f$ and there exist $\mu \in (0,1)$ and $C_{\mathcal{A}^*,\mu}$  independent of $f$ such that $\varphi \in C^{0,\mu}(\bar\Omega)$ and
\begin{equation}
\|\varphi\|_{C^{0,\mu}(\bar\Omega)} \le C_{\mathcal{A}^*,\mu}\|f\|_{L^{\bar p}(\Omega)}\quad \forall f \in L^{\bar p}(\Omega).
\label{E2.6}
\end{equation}
\label{C2.4}
\end{crllr}

\begin{proof}
The first statement follows directly from Theorem \ref{T2.2}. The second is again a consequence of \cite[Theorem 7.3] {Stampacchia65} or \cite[Theorem 14.1]{Lad-Ural-elliptic1968}.
\end{proof}

Under additional assumptions we have the following regularity result.

\begin{thrm}
Suppose Assumption 1 holds and assume further that $a_{ij} \in C^{0,1}(\bar\Omega)$ for $1 \le i,j \le n$,  $a_0 \in L^2(\Omega)$ and $\Gamma$ is of class $C^{1,1}$ or $\Omega$ is convex. Then, $\mathcal{A}:H^2(\Omega) \cap H_0^1(\Omega) \longrightarrow L^2(\Omega)$ is an isomorphism.
\label{T2.5}
\end{thrm}

\begin{proof}
First we observe that $\mathcal{A}:H^2(\Omega) \cap H_0^1(\Omega) \longrightarrow L^2(\Omega)$ is an injective,  continuous linear operator. Indeed, taking into account that $H^2(\Omega) \subset W^{1,\frac{2p}{p-2}}(\Omega)$ and $H^2(\Omega) \subset C(\bar\Omega)$, we get the following estimates
\begin{align*}
&\|\partial_{x_j}(a_{ij}\partial_{x_i}y)\|_{L^2(\Omega)} \le \|\partial_{x_j}a_{ij}\partial_{x_i}y\|_{L^2(\Omega)} + \|a_{ij}\partial^2_{x_i,x_j}y\|_{L^2(\Omega)}\\
&\le \|\partial_{x_j}a_{ij}\|_{L^\infty(\Omega)}\|\partial_{x_i}y\|_{L^2(\Omega)} + \|a_{ij}\|_{L^\infty(\Omega)}\|\partial^2_{x_i,x_j}y\|_{L^2(\Omega)} \le \|a_{ij}\|_{C^{0,1}(\bar\Omega)}\|y\|_{H^2(\Omega)},\\
&\|b\cdot\nabla y\|_{L^2(\Omega)} \le \|b\|_{L^p(\Omega)^n}\|\nabla y\|_{L^{\frac{2p}{p-2}}(\Omega)^n} \le C\|b\|_{L^p(\Omega)^n}\|y\|_{H^2(\Omega)},\\
&\|a_0y\|_{L^2(\Omega)} \le \|a_0\|_{L^2(\Omega)}\|y\|_{L^\infty(\Omega)} \le C\|a_0\|_{L^2(\Omega)}\|y\|_{H^2(\Omega)}.
\end{align*}
The above estimates prove that $\mathcal{A}$ is well defined and continuous. The injectivity of $\mathcal{A}$ is an immediate consequence of Theorem \ref{T2.2}. Let us prove that $\mathcal{A}$ is surjective. Given $u \in L^2(\Omega)$ arbitrary, from Theorem \ref{T2.2} we deduce the existence of an element $y \in H_0^1(\Omega)$ such that $\mathcal{A}y = u$. We have to prove that $y\in H^2(\Omega)$.
To this end we divide the proof into three steps.

\textit{Step 1.-} Here we regularize the coefficients $b$ and $a_0$.
We make the proof for $n=3$ and comment later the modifications for $n=2$.

The Lipschitz regularity of the coefficients $a_{ij}$ implies that $A:H^2(\Omega) \cap H_0^1(\Omega) \longrightarrow L^2(\Omega)$ is an isomorphism; see, for instance, Theorems 2.2.2.3 and 3.2.1.2 of \cite{Grisvard85} for a $C^{1,1}$ boundary $\Gamma$ and a convex domain $\Omega$, respectively. Hence, there exists a constant $C_A$ such that
\begin{equation}\label{E2.7}\|y\|_{H^2(\Omega)} \le C_A\|Ay\|_{L^2(\Omega)}\mbox{ for all }y\in H^2(\Omega).\end{equation}
For $n=3$, we are assuming $p\geq n$, and hence $2p/(p-2)\leq 6$. Therefore, from the Sobolev imbedding theorem, we also know that there exists a constant $M_\Omega$ such that
\begin{equation}\label{E2.8}\|\nabla y\|_{L^{\frac{2p}{p-2}}(\Omega)^n} \le M_\Omega\|y\|_{H^2(\Omega)}\mbox{ for all }y\in H^2(\Omega).\end{equation}
Consider, as in the proof of Lemma \ref{L2.1}, the decomposition  $b=b'+b''$, where now
\[\|b''\|_{L^{p}(\Omega)^{n}}<\varepsilon=\frac{1}{8 C_A M_\Omega}\mbox{ and }\|b'\|_{L^\infty(\Omega)^{n}} < K_{\varepsilon,b}.\]
Consider also two sequences $\{b''_k\}_{k = 1}^\infty \subset L^\infty(\Omega)^{n}$ and $\{a_{0,k}\}_{k = 1}^\infty \subset L^\infty(\Omega)$ such that $b''_k \to b''$ strongly in $L^{p}(\Omega)^{n}$ and $0 \le a_k \to a_0$ strongly in $L^2(\Omega)$. Denote $b_k = b'+b''_k$ and  define the operator $\mathcal{A}_k:H^2(\Omega) \cap H_0^1(\Omega) \longrightarrow L^2(\Omega)$ by
\[
\mathcal{A}_kz = Az + b_k(x) \cdot \nabla z + a_{0,k}(x)z.
\]
Let us prove that $\mathcal A_k \colon H^2(\Omega) \cap H^1_0(\Omega) \to L^2(\Omega)$ is an isomorphism. The proof of the continuity and injectivity follows as we did for $\mathcal{A}$. Now,  from Theorem \ref{T2.2} we deduce the existence of an element $y_k \in H_0^1(\Omega)$ such that $\mathcal{A}_ky_k = u$. This equation can be written as follows
\[
A y_k = u - b_k(x)\cdot\nabla y_k - a_{0,k}(x)y_k.
\]
We have the estimates
\begin{align*}
&\|b_k\cdot\nabla y_k\|_{L^2(\Omega)} \le \|b_k\|_{L^\infty(\Omega)^{n}}\|\nabla y_k\|_{L^2(\Omega)^{n}},\\
&\|a_{0,k}y_k\|_{L^2(\Omega)} \le \|a_{0,k}\|_{L^\infty(\Omega)}\|y_k\|_{L^2(\Omega)}.
\end{align*}
Thus, we have that $Ay_k \in L^2(\Omega)$.  From here and \eqref{E2.7} we infer
\begin{align*}
&\|y_k\|_{H^2(\Omega)} \le C_A\|u - b_k\cdot\nabla y_k - a_{0,k}y_k\|_{L^2(\Omega)}\\
&\le C_A\Big(\|u\|_{L^2(\Omega)} + \|b' \cdot\nabla y_k\|_{L^2(\Omega)} + \|b''_k\cdot\nabla y_k\|_{L^2(\Omega)}+ \|a_{0,k}y_k\|_{L^2(\Omega)}\Big)\\
&\le C_A\Big(\|u\|_{L^2(\Omega)} + \|b'\|_{L^\infty(\Omega)^{n}} \|\nabla y_k\|_{L^2(\Omega)^n}  + \|b''_k\|_{L^{p}(\Omega)^{n}}\|\nabla y_k\|_{L^{\frac{2p}{p-2}}(\Omega)^{n}} + \|a_{0,k}\|_{L^2(\Omega)}\|y_k\|_{L^\infty(\Omega)}\Big).
\end{align*}
From the strong convergences of the sequences $\{b''_k\}_{k = 1}^\infty$ and $\{a_{0,k}\}_{k = 1}^\infty$, we deduce the existence of some integer $k_0$ such that
\[
\|b''_k\|_{L^{p}(\Omega)^{n}} \le 2\|b''\|_{L^{p}(\Omega)^{n}} \text{ and } \|a_{0,k}\|_{L^2(\Omega)} \le 2\|a_0\|_{L^2(\Omega)} \ \ \forall k \ge k_0.
\]
Inserting these inequalities in the above expression and taking into account the relations $\|b'\|_{L^\infty(\Omega)^{n}}\leq K_{\varepsilon,b}$ and $\|b''\|_{L^{p}(\Omega)^{n}}<\varepsilon$, we infer $\forall k \ge k_0$
\begin{equation}\label{E2.9}
\|y_k\|_{H^2(\Omega)} \le C_A\Big(\|u\|_{L^2(\Omega)} + K_{\varepsilon,b} \|\nabla y_k\|_{L^2(\Omega)^{n}} +
2\varepsilon\|\nabla y_k\|_{L^{\frac{2p}{p-2}}(\Omega)^n} + 2\|a_0\|_{L^2(\Omega)}\|y_k\|_{L^\infty(\Omega)}\Big).
\end{equation}
Now, we apply Lions' Lemma %, \cite[Chapter 2, Lemma 6.1]{Necas67},
to the embeddings $H^2(\Omega) \subset H^1(\Omega) \subset L^2(\Omega)$ and $H^2(\Omega) \subset C(\bar\Omega) \subset L^2(\Omega)$, where the first embedding in each of the chains is compact.
Selecting
\[
\lambda = \frac{1}{4C_A\max\{K_{\varepsilon,b},2\|a_0\|_{L^2(\Omega)}\}},
\]
we deduce the existence of a constant $C_\lambda$ such that for all $y \in H^2(\Omega) \cap H_0^1(\Omega)$
\[
\|\nabla y\|_{L^{2}(\Omega)^n} \le \lambda\|y\|_{H^2(\Omega)} + C_\lambda\|y\|_{L^2(\Omega)}
\]
 and
 \[
  \|y\|_{L^\infty(\Omega)}\le \lambda\|y\|_{H^2(\Omega)} + C_\lambda\|y\|_{L^2(\Omega)}.
\]
Inserting these two inequalities in the estimate \eqref{E2.9}, using \eqref{E2.8}, and taking into account the definition of $\varepsilon$, we get
\[
\|y_k\|_{H^2(\Omega)} \le C_A\Big(\|u\|_{L^2(\Omega)} + C_\lambda\left[K_{\varepsilon,b} + 2\|a_0\|_{L^2(\Omega)}\right]\|y_k\|_{L^2(\Omega)}\Big) + \frac{3}{4}\|y_k\|_{H^2(\Omega)} ,
\]
which,  leads to the desired estimate:
\begin{equation}
\|y_k\|_{H^2(\Omega)} \le 4C_A\Big(\|u\|_{L^2(\Omega)} + C_\lambda[K_{\varepsilon,b} + \|a_0\|_{L^2(\Omega)}]\|y_k\|_{L^2(\Omega)}\Big).
\label{E2.10}
\end{equation}

For $n=2$ the proof is slightly different. We take $\{b_k\}_{k = 1}^\infty \subset L^\infty(\Omega)^{n}$ such that $b_k \to b$ strongly in $L^p(\Omega)^{n}$. In the same way as before we obtain that $y_k\in H^2(\Omega)$, but we directly write the estimate
\begin{align*}
&\|y_k\|_{H^2(\Omega)} \le C_A\|u - b_k\cdot\nabla y_k - a_{0,k}y_k\|_{L^2(\Omega)}\\
&\le C_A\Big(\|u\|_{L^2(\Omega)} +  \|b_k\|_{L^p(\Omega)^n}\|\nabla y_k\|_{L^{\frac{2p}{p-2}}(\Omega)^n} + \|a_{0,k}\|_{L^2(\Omega)}\|y_k\|_{L^\infty(\Omega)}\Big).
\end{align*}
Now, we do not have that $\|b_k\|_{L^p(\Omega)^{n}}$ is small, but since for $n=2$ we are assuming $p>n$ we have that $H^2(\Omega)\subset W^{1,\frac{2p}{p-2}}(\Omega)\subset L^2(\Omega)$, the first embedding being compact, and we can argue using Lions' Lemma for the term $\|\nabla y_k\|_{L^{\frac{2p}{p-2}}(\Omega)^{n}}$.

\textit{Step 2.-} Let us prove that $y_k \to y = \mathcal{A}^{-1}u$ strongly  in $H_0^1(\Omega)$. We have
\[
\|y_k -y\|_{H_0^1(\Omega)} \le C_\Omega\|\mathcal{A}_k^{-1} - \mathcal{A}^{-1}\|_{\mathcal{L}(H^{-1}(\Omega),H_0^1(\Omega))}\|u\|_{L^2(\Omega)},
\]
it is enough to show that $\|\mathcal{A}_k^{-1} - \mathcal{A}^{-1}\|_{\mathcal{L}(H^{-1}(\Omega),H_0^1(\Omega))} \to 0$ when $k \to \infty$. First we prove that $\{\mathcal{A}_k^{-1}\}_{k = 1}^\infty$ is bounded in $\mathcal{L}(H^{-1}(\Omega),H_0^1(\Omega))$.  It is immediate that
\begin{equation}
\|\mathcal{A}_k - \mathcal{A}\|_{\mathcal{L}(H_0^1(\Omega),H^{-1}(\Omega))} \le C_{\Omega,p}\|b_k - b\|_{L^{\frac{2p}{p-2}}(\Omega)^n} + C_\Omega\|a_{0.k} - a_0\|_{L^2(\Omega)} \stackrel{k \to \infty}{\longrightarrow} 0.
\label{E2.11}
\end{equation}
Hence, there exists an integer $k_1$ such that
\[
\|\mathcal{A}_k - \mathcal{A}\|_{\mathcal{L}(H_0^1(\Omega),H^{-1}(\Omega))} \le \frac{1}{2\|\mathcal{A}^{-1}\|_{\mathcal{L}(H^{-1}(\Omega),H_0^1(\Omega))}}\ \ \forall k \ge k_1.
\]
This implies that
\[
\|\mathcal{A}^{-1}(\mathcal{A} - \mathcal{A}_k)\|_{\mathcal{L}(H_0^1(\Omega),H_0^1(\Omega))} \le \frac{1}{2}\ \ \forall k\ge k_1.
\]
Hence, the operator $I - \mathcal{A}^{-1}(\mathcal{A} - \mathcal{A}_k)$ is invertible in $H_0^1(\Omega)$ and its inverse is given by
\[
\Big(I - \mathcal{A}^{-1}(\mathcal{A} - \mathcal{A}_k)\Big)^{-1} = \sum_{j = 0}^\infty \Big[\mathcal{A}^{-1}(\mathcal{A} - \mathcal{A}_k)\Big]^j.
\]
From here we obtain for every $k \ge k_1$
\[
\|\Big(I - \mathcal{A}^{-1}(\mathcal{A} - \mathcal{A}_k)\Big)^{-1}\|_{\mathcal{L}(H_0^1(\Omega),H_0^1(\Omega))} \le  \sum_{j = 0}^\infty \|\mathcal{A}^{-1}(\mathcal{A} - \mathcal{A}_k)\|_{\mathcal{L}(H_0^1(\Omega),H_0^1(\Omega))}^j \le 2.
\]
Now, we observe that $\mathcal{A}_k^{-1} = \Big(I - \mathcal{A}^{-1}(\mathcal{A} - \mathcal{A}_k)\Big)^{-1}\mathcal{A}^{-1}$. Thus, we obtain that
\begin{equation}
\|\mathcal{A}_k^{-1}\|_{\mathcal{L}(H^{-1}(\Omega),H_0^1(\Omega))} \le 2\|\mathcal{A}^{-1}\|_{\mathcal{L}(H^{-1}(\Omega),H_0^1(\Omega))}\ \ \forall k \ge k_1.
\label{E2.12}
\end{equation}
Finally, we find with \eqref{E2.11} and \eqref{E2.12} that
\begin{align*}
&\|\mathcal{A}_k^{-1} - \mathcal{A}^{-1}\|_{\mathcal{L}(H^{-1}(\Omega),H_0^1(\Omega))} = \|\mathcal{A}^{-1}[\mathcal{A} - \mathcal{A}_k]\mathcal{A}_k^{-1}\|_{\mathcal{L}(H^{-1}(\Omega),H_0^1(\Omega))}\\
& \le \|\mathcal{A}^{-1}\|_{\mathcal{L}(H^{-1}(\Omega),H_0^1(\Omega))}\|\mathcal{A} - \mathcal{A}_k\|_{\mathcal{L}(H_0^1(\Omega),H^{-1}(\Omega))}\|\mathcal{A}_k^{-1}\|_{\mathcal{L}(H^{-1}(\Omega),H_0^1(\Omega))}\\
&\leq 2\|\mathcal{A}^{-1}\|_{\mathcal{L}(H^{-1}(\Omega),H_0^1(\Omega))}^2\|\mathcal{A} - \mathcal{A}_k\|_{\mathcal{L}(H_0^1(\Omega),H^{-1}(\Omega))} \stackrel{k \to \infty}{\longrightarrow} 0.
\end{align*}

\textit{Step 3.-} Finally, the estimate \eqref{E2.10} and the convergence $y_k \to y$ in $H_0^1(\Omega)$ yields $y_k \rightharpoonup y$ weakly in $H^2(\Omega)$. Since $u \in L^2(\Omega)$ was arbitrary, this implies the surjectivity of $\mathcal{A}$.
\end{proof}

\begin{crllr}
Under the assumptions of Theorem \ref{T2.5} and in addition $\div{b} \in L^2(\Omega)$, the operator $\mathcal{A}^*:H^2(\Omega) \cap H_0^1(\Omega) \longrightarrow L^2(\Omega)$ is an isomorphism.
\label{C2.6}
\end{crllr}

\begin{proof}
Injectivity follows from Corollary \ref{C2.4}. Let us prove that it is surjective. Take $f\in L^2(\Omega)$ and let $\varphi\in H^1_0(\Omega)$ be the unique solution of $\mathcal{A}^*\varphi = f$. Notice that, from the second part of Corollary \ref{C2.4} we also know that $\varphi\in L^\infty(\Omega)$. Taking into account that we can write
\[
\mathcal{A}^*\varphi = A^*\varphi - b(x)\cdot \nabla\varphi + (a_0(x) - \div{b}(x))\varphi,
\]
we have that
\[
A^*\varphi - b(x)\cdot \nabla\varphi + a_0(x)\varphi = f + \div{b}(x)\varphi\mbox{ in }\Omega,\ \varphi = 0\mbox{ on }\Gamma
\]
and the result follows from Theorem \ref{T2.5} because $ f + \div b (x)\varphi\in L^2(\Omega)$.
\end{proof}

\subsection{Analysis of the semilinear equation}
\label{S2.2}
Here, we analyze the equation \eqref{E1.1}. To deal with this equation we make the following hypotheses on the nonlinear term $f$.

\textit{Assumption 2.} We assume that $f:\Omega \times \mathbb{R} \longrightarrow \mathbb{R}$ is a Carath\'eodory function monotone nondecreasing with respect to the second variable satisfying:
\begin{equation}
\forall M > 0 \ \exists \phi_M \in L^{\bar p}(\Omega): |f(x,y)| \le \phi_M(x) \  \mbox{ for a.a. } x \in \Omega \text{ and } \forall |y| \le M.
\label{E2.13}
\end{equation}
Let us recall that $\bar p$ stands for a real number bigger than $\frac{n}{2}$.

Now, we prove the existence and uniqueness of a solution for problem \eqref{E1.1}.

\begin{thrm}
Under the Assumptions 1 and 2, for every $u \in L^{\bar p}(\Omega)$ the equation \eqref{E1.1} has a unique solution $y_u$ in $H_0^1(\Omega) \cap C(\bar\Omega)$. Moreover, there exists a constant $K_f$ independent of $u$ such that
\begin{equation}
\|y_u\|_{H_0^1(\Omega)} + \|y_u\|_{C(\bar\Omega)} \le K_f\Big(\|u\|_{L^{\bar p}(\Omega)} + \|f(\cdot,0)\|_{L^{\bar p}(\Omega)} + 1\Big).
\label{E2.14}
\end{equation}
\label{T2.7}
\end{thrm}

Before proving this theorem we establish the following lemma.

\begin{lmm}
Let $g:\Omega \times \mathbb{R} \longrightarrow \mathbb{R}$ be a function satisfying Assumption 2. We also suppose that Assumption 1 holds. Then, if  $y_1, y_2 \in H_0^1(\Omega) \cap L^\infty(\Omega)$ are solutions of the equations
\begin{equation}
Ay_i + b(x)\cdot\nabla y_i + g(x,y_i) = u_i, \ \ i= 1, 2,
\label{E2.15}
\end{equation}
with $u_1, u_2 \in L^{\bar p}(\Omega)$ and $u_1 \le u_2$ in $\Omega$, then $y_1 \le y_2$ in $\Omega$ as well.
\label{L2.8}
\end{lmm}

\begin{proof}
We make the proof for $n = 3$. The case $n = 2$ can be proved in a similar way. We argue by contradiction, proceeding similarly to the proof of Theorem \ref{T2.2}. If the statement of the lemma is false, then there exists $0 < \rho < \text{esssup}_{x \in \Omega}(y_1(x) - y_2(x))$. Now, we set $z(x) = [(y_1(x) - y_2(x)) - \rho]^+$. We have that $z \in H_0^1(\Omega)$. We denote $\Omega_\rho = \{x \in \Omega : \nabla z(x) \neq 0\}$. Let us observe that we have
\[
\nabla z(x) = \left\{\begin{array}{cl} \nabla (y_1 - y_2)(x) & \text{if } (y_1 - y_2)(x) > \rho,\\0 & \text{otherwise,}\end{array}\right.\]
and
\[ z(x) = 0 \text{ if } (y_1 - y_2)(x) \le \rho.
\]
Using these facts, our assumptions on $b$ and $u_i$, and the monotonicity of $g$ we get
\begin{align*}
&0 \ge \int_\Omega (u_1 - u_2)z\, dx \\
 & = \int_\Omega\Big(\sum_{i, j = 1}^n a_{ij}(x)\partial_{x_i}(y_1 - y_2)\partial_{x_j}z + [b(x)\cdot\nabla (y_1 - y_2)] z \Big) dx\\
& + \int_\Omega[g(x,y_1) - g(x,y_2)]z\, dx \ge \int_{\Omega_\rho}\Big(\sum_{i, j = 1}^n a_{ij}(x)\partial_{x_i}z\partial_{x_j}z + [b(x)\cdot\nabla z] z\Big)\, dx\\
& \ge \Lambda\|\nabla z\|^2_{L^2(\Omega_\rho)^{n}} - \|b\|_{L^3(\Omega_\rho)^{n}}\|\nabla z\|_{L^2(\Omega_\rho)^{n}}\|z\|_{L^6(\Omega_\rho)}.
\end{align*}
Now, we continue as in the proof of Theorem \ref{T2.2} to achieve the contradiction
\[
\|b\|_{L^3(\Omega_\rho)^{n}} \ge \frac{\Lambda}{C_1} > 0\quad  \forall \rho < \text{ess sup}_{x \in \Omega}(y_1 - y_2)(x).
\]
\end{proof}

\begin{proof}[Proof of Theorem \ref{T2.7}] The uniqueness of a solution is an immediate consequence of Lemma \ref{L2.8}. The proof of existence is divided into three steps according to different assumptions on $f$. To simplify the presentation, we redefine $f = f - f(\cdot,0)$ and $u = u - f(\cdot,0) \in L^{\bar p}(\Omega)$. Then, due to Assumption 2, $f$ is continuous and monotone nondecreasing, $f(x,0) = 0$, and $f$ is dominated by a function $\phi_M \in L^{\bar p}(\Omega)$ in $\Omega \times [-M,+M]$ for every $M > 0$.

{\em Step 1.- Assume that there exists $\phi \in L^{\bar p}(\Omega)$ such that $|f(x,y)| \le \phi(x)$ in $\Omega \times \mathbb{R}$}. In this case, we consider the operator $T:C(\bar\Omega) \longrightarrow C(\bar\Omega)$ given by $Tw = y_w$, where $y_w$ is the solution of the problem
\[
\left\{\begin{array}{l} Ay + b(x)\cdot\nabla y + f(x,w) = u \text{ in } \Omega,\\ y = 0\text{ on } \Gamma.\end{array}\right.
\]
From Corollary \ref{C2.3} we have the existence and uniqueness of a solution $y_w \in H_0^1(\Omega) \cap C^{0,\mu}(\bar\Omega)$ for some $\mu \in (0,1)$. This solution satisfies
\[
\|y_w\|_{C^{0,\mu}(\bar\Omega)} \le C_{A,b}\|u-f(x,w)\|_{L^{\bar p}(\Omega)} \le C_{A,b}\Big(\|u\|_{L^{\bar p}(\Omega)} + \|\phi\|_{L^{\bar p}(\Omega)}\Big) =\rho.
\]
From the compactness of the embedding $C^{0,\mu}(\bar\Omega) \subset C(\bar\Omega)$, we deduce that $T$ is a compact operator applying the closed ball $\bar{B}_\rho(0)$ into itself. Hence, from Schauder's fixed point Theorem we infer the existence of a solution $y_u \in H_0^1(\Omega) \cap C(\bar\Omega)$ of \eqref{E1.1}. Moreover, \eqref{E2.14} follows from the above inequality and the redefinition of $u$ along with the fact that $A + b\cdot \nabla I:H_0^1(\Omega) \longrightarrow H^{-1}(\Omega)$ is an isomorphism.

{\em Step 2.- We relax the assumption of step 1 and now we only assume that there exists $\phi \in L^{\bar p}(\Omega)$ such that $f \ge \phi$ in $\Omega \times \mathbb{R}$.} For every integer $k \ge 1$ we set $f_k(x,y) = f(x,\min\{y,k\})$. Then, from \eqref{E2.13} we infer
\[
\phi(x) \le f_k(x,y) \le f(x,\text{proj}_{[-k,+k]}(y)) \le \phi_k(x).
\]
Hence, $|f_k(x,y)| \le \psi_k(x) = \max\{|\phi(x)|,|\phi_k(x)|\}$ with $\psi_k \in L^{\bar p}(\Omega)$. Then, we can apply Step 1 to deduce the existence of a function $y_k \in H_0^1(\Omega) \cap C(\bar\Omega)$ satisfying
\[
\left\{\begin{array}{l} Ay_k + b(x)\cdot\nabla y_k + f_k(x,y_k) = u \text{ in } \Omega,\\ y_k = 0\text{ on } \Gamma.\end{array}\right.
\]
Now, by Corollary \ref{C2.3} there exists a function $y \in H_0^1(\Omega) \cap C(\bar\Omega)$ solution of
\[
\left\{\begin{array}{l} Ay + b(x)\cdot\nabla y  = u - \phi \text{ in } \Omega,\\ y = 0\text{ on } \Gamma.\end{array}\right.
\]
Subtracting both equations we get
\[
A(y_k - y) + b(x)\cdot\nabla(y_k - y) = -f_k(x,y_k) + \phi \le 0.
\]
Then, Lemma \ref{L2.8} implies that $y_k \le y$ in $\Omega$. Therefore, if we take $k > \|y\|_{C(\bar\Omega)}$, we have that $f_k(x,y_k) = f(x,y_k)$, and therefore $y_k$ is solution of \eqref{E1.1}. The estimate \eqref{E2.14} follows from the bound for $y_k$ independently of $k$ and Assumption 2.

{\em Step 3.- The general case.} Let us define $f_k(x,y) = f(x,\text{proj}_{[-k,+k]}(y))$. Then, according to Assumption 2, there exists a function $\phi_k \in L^{\bar p}(\Omega)$ such that $|f_k(x,y)| \le \phi_k(x)$ in $\Omega \times \mathbb{R}$. Therefore, from Step 1 we know that there exists $y_k \in H_0^1(\Omega) \cap C(\bar\Omega)$ satisfying
\[
Ay_k + b(x)\cdot\nabla y_k + f_k(x,y_k) = u.
\]
Now, we take $z_1 \in H_0^1(\Omega) \cap C(\bar\Omega)$ satisfying $Az_1 + b(x)\cdot\nabla z_1 + f(x,z_1^+) = u$. The existence of such a function follows from Step 2 because $f(x,z_1^+) \ge 0$. This equation can be written as
\[
Az_1 + b(x)\cdot\nabla z_1 + f_k(x,z_1) = u + f_k(x,z_1)-f(x,z_1^+).
\]
From the monotonicity of $f$ with respect to the second variable, we have that $f_k(x,z_1) - f(x,z_1^+) \le 0$, and hence Lemma \ref{L2.8} implies that $z_1 \le y_k$ in $\Omega$.

Now, let $z_2 \in H_0^1(\Omega) \cap C(\bar\Omega)$ satisfy
\[
Az_2 + b(x)\cdot\nabla z_2 = u - f(x,-\|z_1\|_{C(\bar\Omega)}).
\]
The existence of such a function follows from Theorem \ref{T2.2} and Corollary \ref{C2.3}. Indeed, it is enough to observe that from \eqref{E2.13} with $M =  \|z_1\|_{C(\bar\Omega)}$ we deduce the existence of a function $\phi_M \in L^{\bar p}(\Omega)$ such that  $|f(x,-\|z_1\|_{C(\bar\Omega)})| \le \phi_M(x)$ for almost all $x \in \Omega$. Then, $ f(\cdot,-\|z_1\|_{C(\bar\Omega)}) \in L^{\bar p}(\Omega)$ holds.
Noticing that the equation satisfied by $y_k$ can be written
\[Ay_k + b(x)\cdot\nabla y_k  = u - f_k(x,y_k)\]
and that, using the fact that $y_k \ge z_1 \ge -\|z_1\|_{C(\bar\Omega)}$, we have $f_k(x,y_k(x)) \ge f(x,-\|z_1\|_{C(\bar\Omega)})$.
Then, subtracting the equations satisfied by $y_k$ and $z_2$ we get
\[
A(z_2 - y_k) + b(x)\cdot\nabla(z_2 - y_k) = f_k(x,y_k) - f(x,-\|z_1\|_{C(\bar\Omega)}) \ge 0.
\]
Therefore, $z_1 \le y_k \le z_2$ and, hence, $f_k(x,y_k) = f(x,y_k)$ for every $k > \max\{\|z_1\|_{C(\bar\Omega)},\|z_2\|_{C(\bar\Omega)}\}$, and consequently $y_k$ is solution of \eqref{E1.1} for $k$ large enough. As in the previous step, the estimate \eqref{E2.14} follows from the bound for $y_k$ independently of $k$.
\end{proof}

Now, we establish some additional regularity for the solutions of \eqref{E1.1}.

\begin{crllr}
There exists some $\mu \in (0,1)$ such that the solution $y_u$ of \eqref{E1.1} belongs to $C^{0,\mu}(\bar\Omega)$. Moreover, for every $M > 0$ there exists a constant $K_{f,\mu,M}$ such that
\[
\|y_u\|_{C^{0,\mu}(\bar\Omega)} \le K_{f,\mu,M}\quad \forall u \in L^{\bar p}(\Omega) \text{ satisfying }\|u\|_{L^{\bar p}(\Omega)}\le M.
\]
\label{C2.9}
\end{crllr}

\begin{proof}
Since $Ay_u + b\cdot\nabla y_u = u - f(x,y_u)$, this corollary follows from \eqref{E2.13}, \eqref{E2.14} and Corollary \ref{C2.3}.
\end{proof}

\begin{thrm}
Suppose that Assumption 1 holds, $a_{ij} \in C^{0,1}(\bar\Omega)$ for $1 \le i,j \le n$ and $a_0\in L^2(\Omega)$. We also suppose that Assumption 2 holds with $\bar p = 2$, and that $\Gamma$ is of class $C^{1,1}$ or $\Omega$ is convex. Then, for every $u \in L^2(\Omega)$, the equation \eqref{E1.1} has a unique solution $y_u \in H^2(\Omega) \cap H_0^1(\Omega)$. Moreover, for every $M > 0$ there exists a constant $C_{\mathcal{A},f,M}$ such that
\[
\|y_u\|_{H^2(\Omega)} \le C_{\mathcal{A},f,M}\quad \forall u \in L^{\bar p}(\Omega) \text{ satisfying }\|u\|_{L^{\bar p}(\Omega)}\le M.
\]
\label{T2.10}
\end{thrm}

This is an immediate consequence of Theorems \ref{T2.5} and \ref{T2.7}. The following result on the continuous dependence of the state $y_u$ respect to $u$ will be useful to prove the existence of a solution for the control problem \Pb.

\begin{thrm}
Let $\{u_k\}_{k = 1}^\infty \subset L^{\bar p}(\Omega)$ be a sequence weakly converging to $u$ in $L^{\bar p}(\Omega)$. Then, under the assumptions of Theorem \ref{T2.7} we have that $y_{u_k} \to y_u$ strongly in $H_0^1(\Omega) \cap C(\bar\Omega)$.
\label{T2.11}
\end{thrm}

\begin{proof}
From Theorem \ref{T2.7} and Corollary \ref{C2.9} we know that $\{y_{u_k}\}_{k = 1}^\infty$ is bounded in $H_0^1(\Omega) \cap C^{0,\mu}(\bar\Omega)$. Hence, using the compactness of the embedding $C^{0,\mu}(\bar\Omega) \subset C(\bar\Omega)$, we deduce the existence of a subsequence, denoted in the same way, and an element $y \in H_0^1(\Omega) \cap C(\bar\Omega)$ such that  $y_{u_k} \rightharpoonup y$ in $H_0^1(\Omega)$ and $y_{u_k} \to y$ in $C(\bar\Omega)$. Now, setting $M = \max_{1 \le k <\infty}\|y_{u_k}\|_{C(\bar\Omega)}$, we deduce from \eqref{E2.13}
\begin{align*}
&|f(x,y_{u_k}(x))| \le |f(x,0)| + |f(x,y_{u_k}(x)) - f(x,0)|\\
& \le |f(x,0)| + C_{f,M}|y_{u_k}(x)| \le |f(x,0)| + C_{f,M}M.
\end {align*}
 As a consequence we have that $f(x,y_{u_k}) \to f(x,y)$ strongly in $L^{\bar p}(\Omega)$. Moreover, the compactness of the embedding $L^{\bar p}(\Omega) \subset H^{-1}(\Omega)$ implies that $u_k \to u$ strongly in $H^{-1}(\Omega)$. Hence, from Theorem \ref{T2.2} we infer that $y_{u_k} = (A + b\cdot\nabla I)^{-1}(u_k - f(x,y_{u_k})) \to (A + b\cdot\nabla I)^{-1}(u - f(x,y))$ strongly in $H_0^1(\Omega)$. Thus,  we have that $y = y_u$ and $y_{u_k} \to y_u$ strongly in $H_0^1(\Omega) \cap C(\bar\Omega)$. Since, this convergence holds for any converging subsequence, we deduce that the whole sequence converges as indicated in the statement of the theorem.
\end{proof}

To finish this section we analyze the differentiability of the relation $u \to y_u$. To this end, we make the following assumptions on $f$.

\textit{Assumption 3.} We assume that $f:\Omega \times \mathbb{R}\longrightarrow \mathbb{R}$ is a Carath\'eodory function of class $C^2$ with respect to the second variable satisfying:
\begin{equation}
f(\cdot,0)\in L^{\bar p}(\Omega)  \text{ and } \frac{\partial f}{\partial y}(x,y) \geq 0 \  \mbox{ a.e. in } \Omega \text{ and } \forall y \in \mathbb{R},
\label{E2.16}
\end{equation}
and for all $M>0$ there exists a constant $C_{f,M}>0$ such that
\begin{equation}
\left|\frac{\partial f}{\partial y}(x,y)\right|+
\left|\frac{\partial^2 f}{\partial y^2}(x,y)\right| \leq C_{f,M}
\mbox{ for a.e. } x \in \Omega \mbox{ and for all } |y| \leq M.
\label{E2.17}
\end{equation}
For every $M > 0$ and $\varepsilon > 0$ there exists $\delta > 0$, depending on $M$ and $\varepsilon$, such that
\begin{equation}
\left|\frac{\partial^2f}{\partial y^2}(x,y_2) -
\frac{\partial^2f}{\partial y^2}(x,y_1)\right| < \varepsilon \
\mbox{ if } |y_1|, |y_2| \leq M,\ |y_2 - y_1| \le \delta, \mbox{ for a.a. } x \in
\Omega.
\label{E2.18}
\end{equation}
Let us recall again that $\bar p$ stands for a real number bigger than $\frac{n}{2}$. It is obvious that Assumption 3 implies Assumption 2. Therefore, all the previous results remain valid if we replace Assumption 2 by Assumption 3.

\begin{thrm}
Let us suppose that Assumptions 1 and 3 hold. Then, the mapping $G:L^{\bar p}(\Omega) \longrightarrow H_0^1(\Omega) \cap C(\bar\Omega)$ given by $G(u) = y_u$ is well defined and of class $C^2$. Moreover, given $u, v \in L^{\bar p}(\Omega)$, $z_v = DG(u)v$ is the solution of
\begin{equation}
\left\{\begin{array}{l}\displaystyle Az + b(x)\cdot\nabla z + \frac{\partial f}{\partial y}(x,y_u)z = v \text{ in } \Omega,\\ z = 0\text{ on } \Gamma.\end{array}\right.
\label{E2.19}
\end{equation}
For $v_1, v_2 \in L^{\bar p}(\Omega)$ the second derivative $z_{v_1,v_2} = D^2G(u)(v_1,v_2)$ is the solution of the equation
\begin{equation}
\left\{\begin{array}{l}\displaystyle Az + b(x)\cdot\nabla z + \frac{\partial f}{\partial y}(x,y_u)z = -\frac{\partial^2f}{\partial y^2}(x,y_u)z_{v_1}z_{v_2} \text{ in } \Omega,\\ z = 0\text{ on } \Gamma,\end{array}\right.
\label{E2.20}
\end{equation}
where $z_{v_i} = DG(u)v_i$, $i = 1,2$.
\label{T2.12}
\end{thrm}

\begin{proof}
The fact that $G$ is well defined is a straightforward consequence of Theorem \ref{T2.7}. To prove the differentiability we will use the implicit function theorem as follows. We consider the vector space
\[
Y = \{y \in H_0^1(\Omega) \cap C(\bar\Omega) : Ay + b\cdot\nabla y \in L^{\bar p}(\Omega)\}.
\]
This is a Banach space when we endow it with the norm
\[
\|y\|_Y = \|y\|_{H_0^1(\Omega)} + \|y\|_{C(\bar\Omega)} + \|Ay+b\cdot\nabla y\|_{L^{\bar p}(\Omega)}.
\]
Let us consider the mapping
\[
\mathcal{F}:Y \times L^{\bar p}(\Omega) \longrightarrow L^{\bar p}(\Omega)
\]
given by
\[
\mathcal{F}(y,u) = Ay + b(x)\cdot\nabla y + f(\cdot, y) - u.
\]
Using Assumption 3 and Corollary \ref{C2.3} it is easy to check that $f(\cdot,y) \in L^{\bar p}(\Omega)$ for every $y \in Y$, and the mapping $y \in Y \to f(\cdot,y) \in L^{\bar p}(\Omega)$ is of class $C^2$.
Hence, $\mathcal{F}$ is well defined and it is of class $C^2$. Moreover, the linear mapping
\[
\frac{\partial\mathcal{F}}{\partial y}(y,u):Y \longrightarrow L^{\bar p}(\Omega)
\]
\[
\frac{\partial\mathcal{F}}{\partial y}(y,u)z = Az + b(x)\cdot\nabla z + \frac{\partial f}{\partial y}(x,y)z
\]
is an isomorphism. Indeed, if we consider the operator $\mathcal{A}$ defined by \eqref{E2.2} with $a_0(x) = \frac{\partial f}{\partial y}(x,y(x))$, we have to prove that $\mathcal{A}:Y \longrightarrow L^{\bar p}(\Omega)$ is an isomorphism. From the definition of $Y$ and the above estimates, we know that $\mathcal{A}$ is well defined and continuous. From Theorem \ref{T2.2} we also deduce the existence of a unique solution $z \in H_0^1(\Omega)$ of the equation $\mathcal{A}z = v$ for every $v \in L^{\bar p}(\Omega) \subset H^{-1}(\Omega)$. In addition, from Corollary \ref{C2.3} we know that $z \in C(\bar\Omega)$. Hence, we have that $z \in Y$ and $\mathcal{A}$ is an isomorphism. Then, we can apply the implicit function theorem and deduce easily the theorem; see e.g. \cite[Proposition 16]{Casas-Mateos2017}.
\end{proof}

\section{Analysis of the optimal control problem}
\label{S3}
\setcounter{equation}{0}

In this section, we firstly prove the existence of a global solution $\bar u$ of the control problem. Then, we derive first and second order necessary optimality conditions for local solutions. Finally, we prove sufficient conditions for local optimality. In the whole section we suppose that Assumptions 1 and 3 are fulfilled.

\begin{thrm}
Let us assume that $L:\Omega \times \mathbb{R} \longrightarrow \mathbb{R}$ is a Carath\'edory function satisfying
\begin{equation}
\forall M > 0 \ \exists \psi_M \in L^1(\Omega) : |L(x,y)| \le \psi_M(x)\ \text{ for a.a. } x \in \Omega \text{ and } \forall |y| \le M.
\label{E3.1}
\end{equation}
Then, if $\uad$ is bounded in $L^2(\Omega)$ or $L$ is bounded from below, the control problem \Pb has at least one solution $\bar u$.
\label{T3.1}
\end{thrm}

\begin{proof}
Let $\{u_k\}_{k = 1}^\infty \subset \uad$ be a minimizing sequence of \Pb. From the boundedness of $\uad$ or the lower boundedness of $L$ we deduce that $\{u_k\}_{k = 1}^\infty$ is bounded in $L^2(\Omega)$. Hence, we can take a subsequence, denoted in the same way, converging weakly in $L^2(\Omega)$ to some element $\bar u$. Since $\uad$ is weakly closed in $L^2(\Omega)$ we infer that $\bar u \in \uad$. Moreover, Theorem 2.5 implies that $y_{u_k} \to y_{\bar u}$ strongly in $H_0^1(\Omega) \cap C(\bar\Omega)$. Therefore, using the assumption \eqref{E3.1} along with Lebesgue's dominated convergence theorem, we get that $J(\bar u) \le \liminf_{k \to \infty}J(u_k) = \inf\Pb$ and, hence, $\bar u$ is a solution of \Pb.
\end{proof}

Before establishing the optimality conditions for \Pb, we study the differentiability of $J$. To this end we make the following assumptions on $L$.

\textit{Assumption 4.} We assume that $L:\Omega \times \mathbb{R}\longrightarrow \mathbb{R}$ is a Carath\'eodory function of class $C^2$ with respect to the second variable satisfying that $L(\cdot,0)\in L^1(\Omega)$ and for all $M>0$ there exist a function $\psi_M \in L^{\bar p}(\Omega)$ with $\bar p > \frac{n}{2}$ and a constant $C_{L,M}>0$ such that
\begin{equation}
\left|\frac{\partial L}{\partial y}(x,y)\right| \le \psi_M(x) \ \text{ and }\ \left|\frac{\partial^2 L}{\partial y^2}(x,y)\right| \leq C_{L,M} \mbox{ for a.e. }x \in \Omega\mbox{ and for all }|y| \leq M.\label{E3.2}
\end{equation}
In addition, for every $M > 0$ and $\varepsilon > 0$ there exists $\delta > 0$, depending on $M$ and $\varepsilon$, such that
\begin{equation}
\left|\frac{\partial^2L}{\partial y^2}(x,y_2) -
\frac{\partial^2L}{\partial y^2}(x,y_1)\right| < \varepsilon \
\mbox{ if } |y_1|, |y_2| \leq M,\ |y_2 - y_1| \le \delta, \mbox{ for a.a. } x \in
\Omega.
\label{E3.3}
\end{equation}
It is obvious that \eqref{E3.1} holds under Assumption 4. In the rest of the paper, we will suppose that Assumptions 1, 2 and  4 are fulfilled. Then, we have the following differentiability result.

\begin{thrm}
The functional $J$ is of class $C^2$. Moreover, given $u, v, v_1, v_2 \in L^2(\Omega)$ we have
\begin{align}
&J'(u)v = \int_\Omega (\varphi_u + \nu u)v\, dx,\label{E3.4}\\
&J''(u)(v_1,v_2) = \int_\Omega\Big[\frac{\partial^2L}{\partial y^2}(x,y_u) - \varphi_u\frac{\partial^2f}{\partial y^2}(x,y_u)\Big]z_{v_1}z_{v_2}\, dx + \nu\int_\Omega v_1v_2\, dx,\label{E3.5}
\end{align}
where $\varphi_u \in H_0^1(\Omega) \cap C(\bar\Omega)$ is the unique solution of the adjoint equation
\begin{equation}
\left\{\begin{array}{l}\displaystyle A^*\varphi - \div[b(x)\varphi] + \frac{\partial f}{\partial y}(x,y_u)\varphi = \frac{\partial L}{\partial y}(x,y_u) \text{ in } \Omega,\\ \varphi = 0\text{ on } \Gamma.\end{array}\right.
\label{E3.6}
\end{equation}
\label{T3.2}
\end{thrm}

\begin{proof}
The $C^2$ differentiability of $J$ is an immediate consequence of Theorem \ref{T2.12}, Assumption 4 and the chain rule. Moreover, the derivation of the formulas \eqref{E3.4} and \eqref{E3.5} is standard.
The existence of a unique $\varphi_u \in H_0^1(\Omega) \cap C(\bar\Omega)$ follows from Corollary \ref{C2.4} and the facts that \eqref{E2.17} along with $y_u \in C(\bar\Omega)$ implies that $a_0 = \frac{\partial f}{\partial y}(\cdot,y_u) \in L^\infty(\Omega)$ and assumption \eqref{E3.2} implies that $\frac{\partial L}{\partial y}(\cdot,y_u) \in L^{\bar p}(\Omega)$.
\end{proof}

Since \Pb is not a convex problem, we consider local solutions of \Pb as well. Let us state precisely the different concepts of local solution.

\begin{dfntn}
We say that $\bar u\in \uad$ is an $L^r(\Omega)$-weak local minimum of \Pb with $r\in [1,+\infty]$, if there exists some $\varepsilon >0$ such that
\[
J(\bar u)\leq J(u)\quad\forall u\in \uad\mbox{ with }\|\bar u-u\|_{L^r(\Omega)}\leq \varepsilon.
\]
An element $\bar u\in\uad$ is said a strong local minimum of \Pb if there exists some $\varepsilon>0$ such that
\[
J(\bar u)\leq J(u)\quad \forall u\in \uad\mbox{ with }\|y_{\bar u}-y_u\|_{L^\infty(\Omega)}  \leq \varepsilon.
\]
We say that $\bar u\in \uad$ is a strict (weak or strong) local minimum if the above inequalities are strict for $u\neq \bar
u$.
\label{D3.3}
\end{dfntn}

As far as we know, the notion of strong local solutions in the framework of control theory was introduced in \cite{BBS2014} for the first time; see also \cite{BS2016}. We analyze the relationships among these concepts in the followin lemma.

\begin{lmm}
The following properties  hold:

If $\uad$ is bounded in $L^2(\Omega)$, then
\begin{enumerate}
\item $\bar u$ is an $L^1(\Omega)$-weak local minimum of \Pb if and only if it is an $L^r(\Omega)$-weak local minimum of \Pb for every $r\in (1,+\infty)$.
\item If $\bar u$ is an $L^r(\Omega)$-weak local minimum of \Pb for some $r < +\infty$, then it is an $L^\infty(\Omega)$-weak local minimum of \Pb.
\item $\bar u$ is a strong local minimum of \Pb if and only if it is an $L^r(\Omega)$-weak local minimum of \Pb for all $r \in [1,\infty)$.
\end{enumerate}

If $\uad$ is not bounded in $L^2(\Omega)$, then
\begin{enumerate}
\item If $\bar u$ is an $L^2(\Omega)$-weak local solution and $L$ is bounded from below, then $\bar u$ is an $L^1(\Omega)$-weak local solution.

\item If $\bar u$ is an $L^p(\Omega)$-weak local solution, then $\bar u$ is an $L^q(\Omega)$-weak local solution for every $p  < q \le \infty$.

\item $\bar u$ is an $L^2(\Omega)$-weak local solution if and only if it is a strong local solution.
\end{enumerate}
\label{L3.4}
\end{lmm}

The reader is referred to \cite{Casas-Mateos2020} and \cite{Casas-Troltzsch2019B}  for the proof of this lemma. To deduce that any strong local solution is an $L^2(\Omega)$-weak local solution the following estimate is used
\[
\|y_u - y_{\bar u}\|_{C(\bar\Omega)} \le C\|u - \bar u\|_{L^2(\Omega)} \quad \forall u \in B_\varepsilon(\bar u),
\]
where $B_\varepsilon(\bar u)$ is the ball in $L^2(\Omega)$. This inequality follows from the next result.

\begin{lmm}
Let $\mathcal{U}$ be a bounded subset of $L^{\bar p}(\Omega)$. Then, there exists a constant $M_\mathcal{U}$ such that
\[
\|y_u - y_v\|_{H_0^1(\Omega)} + \|y_u - y_v\|_{C(\bar\Omega)} \le M_\mathcal{U}\|u - v\|_{L^{\bar p}(\Omega)}\ \ \forall u \in \mathcal{U}.
\]
\label{L3.5}
\end{lmm}

\begin{proof}
Without loss of generality, we can suppose that $\mathcal{U}$ is convex. Otherwise, we replace it by its convex hull, which is also a bounded set.
Given $u,v\in\mathcal{U}$, from Theorem \ref{T2.12} and the mean value theorem we have
\[
\|y_u - y_v\|_{H_0^1(\Omega)} + \|y_u - y_v\|_{C(\bar\Omega)} \le \sup_{\hat u \in \mathcal{U}}\|DG(\hat u)\|_{\mathcal{L}(L^{\bar p}(\Omega),H_0^1(\Omega)\cap C(\bar\Omega))}\|u - v\|_{L^{\bar p}(\Omega)}.
\]
Then, it is enough to prove that $\|DG(\hat u)\|_{\mathcal{L}(L^{\bar p}(\Omega),H_0^1(\Omega)\cap C(\bar\Omega))}$ is bounded by a constant $M_\mathcal{U}$ for every $\hat u \in \mathcal{U}$. From Corollary \ref{C2.9}, we know that $M = \sup\{\|y_u\|_{C(\bar\Omega)} : u \in \mathcal{U}\}<+\infty$. Hence, from Assumption 3 we have
\[
\left|\frac{\partial f}{\partial y}(x,y_{u}(x))\right| \le C_{f,M}\ \text{ for a.a. } x \in \Omega \text{ and } \forall u \in \mathcal{U}.
\]
Now, given $u \in \mathcal{U}$ arbitrary and $v \in L^{\bar p}(\Omega)$ with $\|v\|_{L^{\bar p}(\Omega)} = 1$, we denote by $z$ and $z_0$ the elements of $H_0^1(\Omega) \cap C(\bar\Omega)$ satisfying the equations
\begin{align*}
&\left\{\begin{array}{l}\displaystyle Az + b(x)\cdot\nabla z + \frac{\partial f}{\partial y}(x,y_{u})z = v \text{ in } \Omega,\\ z = 0\text{ on } \Gamma.\end{array}\right.\\
&\left\{\begin{array}{l}\displaystyle Az_0 + b(x)\cdot\nabla z_0  = |v| \text{ in } \Omega,\\ z = 0\text{ on } \Gamma.\end{array}\right.
\end{align*}
Taking in Lemma \ref{L2.8} $g = 0$, $u_1 = 0$ and $u_2 = |v|$, we deduce that $z_0 \ge 0$. Now, subtracting and adding the equations satisfied by $z$ and $z_0$, and using the monotonicity of $f$ we get
\begin{align*}
&A(z_0 - z) + b(x)\cdot\nabla(z_0 - z) + \frac{\partial f}{\partial y}(x,y_{u})(z_0 - z) = \frac{\partial f}{\partial y}(x,y_{u})z_0 + |v| - v \ge 0 \text{ in } \Omega,\\
&A(z_0 + z) + b(x)\cdot\nabla(z_0 + z) + \frac{\partial f}{\partial y}(x,y_{u})(z_0 + z) = \frac{\partial f}{\partial y}(x,y_{u})z_0 + |v| + v \ge 0 \text{ in } \Omega.
\end{align*}
Using again Lemma \ref{L2.8} we infer that $z_0 - z \ge 0$ and $z_0 + z \ge 0$ in $\Omega$, or equivalently $-z_0 \le z \le z_0$ in $\Omega$. Thus, we get with \eqref{E2.4}
\[
\|z\|_{C(\bar\Omega)} \le \|z_0\|_{C(\bar\Omega)} \le C_{A,b}\|v\|_{L^p(\Omega)} = C_{A,b}.
\]

On the other hand, from Theorem \ref{T2.2} we deduce the existence of a constant $C$ independent of $v$ such that
\begin{align*}
&\|z\|_{H_0^1(\Omega)} \le C\left(\|v\|_{L^p(\Omega)} + \left\|\frac{\partial f}{\partial y}(x,y_{u})z\right\|_{L^p(\Omega)}\right)\\
&\le C\left(1 + \left\|\frac{\partial f}{\partial y}(x,y_{u})\right\|_{L^\infty(\Omega)}\|z\|_{C(\bar\Omega)}|\Omega|^{\frac{1}{p}}\right) \le C(1 + C_{f,M}C_{A,b}|\Omega|^{\frac{1}{p}}).
\end{align*}

Hence, we have
\[
\|z\|_{H_0^1(\Omega)} + \|z\|_{C(\bar\Omega)} \le C(1 + C_{f,M}C_{A,b}|\Omega|^{\frac{1}{p}}) + C_{A,b} = M_{\mathcal{U}}.
\]
Since $u$ and $v$ are arbitrary, we conclude that $\|DG(\hat u)\|_{\mathcal{L}(L^{\bar p}(\Omega),H_0^1(\Omega)\cap C(\bar\Omega))} \le M_\mathcal{U}$, and the lemma follows.
\end{proof}

Now, we establish the first order optimality conditions.

\begin{thrm}
Let $\bar u$ be a local solution of \Pb in any of the previous senses. Then there exist two unique elements $\bar y, \bar\varphi \in H_0^1(\Omega) \cap C(\bar\Omega)$ such that
\begin{align}
& \left\{\begin{array}{l} A\bar y + b(x)\cdot\nabla\bar y + f(x,\bar y) = \bar u \text{ in } \Omega,\\ \bar y = 0\text{ on } \Gamma,\end{array}\right.
\label{E3.7}\\
&\left\{\begin{array}{l}\displaystyle A^*\bar\varphi - \div[b(x)\bar\varphi] + \frac{\partial f}{\partial y}(x,\bar y)\bar\varphi = \frac{\partial L}{\partial y}(x,\bar y) \text{ in } \Omega,\\ \bar\varphi = 0\text{ on } \Gamma,\end{array}\right.
\label{E3.8}\\
&\int_\Omega(\bar\varphi + \nu \bar u)(u - \bar u)\, dx \ge 0 \quad \forall u \in \uad. \label{E3.9}
\end{align}
\label{T3.6}
\end{thrm}

This theorem is consequence of the expression for $J'$ given in \eqref{E3.4} and the convexity of $\uad$, which implies that $J'(\bar u)(u - \bar u) \ge 0$ holds for every $u \in \uad$. As a consequence of this theorem we have the following regularity result on the optimal control.

\begin{crllr}
Let $\bar u$ satisfy \eqref{E3.7}--\eqref{E3.9} along with $(\bar y,\bar\varphi)$, then $\bar u \in H^1(\Omega) \cap C(\bar\Omega)$ holds. Moreover, if $a_{ij} \in C^{0,1}(\bar\Omega)$ for $1 \le i,j \le n$, $\div{b} \in L^2(\Omega)$, $\bar p = 2$, and $\Gamma$ is of class $C^{1,1}$ or $\Omega$ is convex, then $\bar y, \bar\varphi \in H^2(\Omega) \cap H_0^1(\Omega)$ holds. Finally, if $\uad = L^2(\Omega)$, then we have that $\bar u \in H^2(\Omega) \cap H_0^1(\Omega)$.
\label{C3.7}
\end{crllr}

\begin{proof}
It is well known that \eqref{E3.9} implies that
\[
\bar u(x) = \text{\rm Proj}_{[\alpha,\beta]}\big(-\frac{1}{\nu}\bar\varphi(x)\big).
\]
Then, the $H^1(\Omega) \cap C(\bar\Omega)$ regularity of $\bar u$ follows from this formula and the same regularity of $\bar\varphi$. Under the additional assumptions on the data of the problem, the regularity of $\bar y$ and $\bar\varphi$ follows from Theorem \ref{T2.10} and Corollary \ref{C2.6}. Finally, if $\uad = L^2(\Omega)$, then \eqref{E3.9} is reduced to $\bar\varphi + \nu\bar u = 0$, hence $\bar u$ enjoys the same regularity as $\bar\varphi$.
\end{proof}

In order to write the second order optimality conditions we introduce the cone of critical directions. Let $\bar u \in \uad$ be a function satisfying the system \eqref{E3.7}-\eqref{E3.9} along with the associated state $\bar y$ and adjoint state $\bar\varphi$. We define the cone
\[
C_{\bar u} = \{v \in L^2(\Omega) : J'(\bar u)v = 0 \text{ and  \eqref{E3.10} holds}\}
\]
\begin{equation}
  v(x,t)\left\{\begin{array}{cl}
                    \geq 0 & \mbox{ if }\bar u(x,t)=\alpha, \\
                    \leq 0 & \mbox{ if }\bar u(x,t)=\beta.
                  \end{array}\right.
\label{E3.10}
\end{equation}
Let us observe that \eqref{E3.9} implies that
\[
\bar\varphi(x) + \nu\bar u(x)\left\{\begin{array}{cl} \ge 0&\text{if } \bar u(x) = \alpha,\\\le 0&\text{if } \bar u(x) = \beta.\end{array}\right.
\]
Therefore, if $v \in L^2(\Omega)$ satisfies \eqref{E3.10}, then $J'(\bar u)v \ge 0$ holds, and $J'(\bar u)v = 0$ if and only if $v(x) = 0$ if $\bar\varphi(x) + \nu\bar u(x) \neq 0$.

In the case where there are not control constraints, namely $\uad = L^2(\Omega)$, then $J'(\bar u) = 0$ and  $C_{\bar u} = L^2(\Omega)$.

Now, we have the second order conditions.

\begin{thrm}
If $\bar u$ is a local solution of \Pb in any sense of those given in Definition \ref{D3.3}, then $J''(\bar u)v^2 \ge 0$ $\forall v \in C_{\bar u}$. Conversely, if $\bar u \in \uad$ satisfies \eqref{E2.12}--\eqref{E2.14} along with $(\bar y,\bar\varphi)$ and
\begin{equation}
J''(\bar u)v^2 > 0\quad \forall v \in C_{\bar u}\setminus\{0\},
\label{E3.11}
\end{equation}
then there exist $\varepsilon > 0$ and $\kappa > 0$ such that
\begin{equation}
J(\bar u) + \frac{\kappa}{2}\|u - \bar u\|^2_{L^2(\Omega)} \le J(u) \ \ \forall u \in \uad : \|y_u - \bar y\|_{L^\infty(\Omega)} \leq \varepsilon.
\label{E3.12}
\end{equation}
\label{T3.8}
\end{thrm}

\begin{proof}
The proof follows the steps of \cite{Casas-Troltzsch2012} or \cite{Casas-Troltzsch2014A}. To reproduce that proof we have to use that $y_{u_k} \to y_u$ and $\varphi_{u_k} \to \varphi_u$ strongly in $H_0^1(\Omega) \cap C(\bar\Omega)$ when $u_k \rightharpoonup u$ in $L^2(\Omega)$. The convergence for the states is proved in Theorem \ref{T2.11}. Here we prove the part corresponding to the adjoint states. To this end we set
\begin{align*}
&\mathcal{A}^*\varphi = A^*\varphi - \div[b(x)\varphi] + \frac{\partial f}{\partial y}(x,y_u)\varphi,\\
&\mathcal{A}_k^*\varphi = A^*\varphi - \div[b(x)\varphi] + \frac{\partial f}{\partial y}(x,y_{u_k})\varphi.
\end{align*}
Since $y_{u_k} \to y_u$ in $C(\bar\Omega)$, there exists $M > 0$ such that $\|y_{u_k}\|_{C(\bar\Omega)} \le M$ $\forall k$. Then, from \eqref{E2.17} and the mean value theorem we deduce for $\|\varphi\|_{H_0^1(\Omega)} \le 1$:
\begin{align*}
&\|(\mathcal{A}^* - \mathcal{A}^*_k)\varphi\|_{H^{-1}(\Omega)} \le C_\Omega\Big\|\Big[\frac{\partial f}{\partial y}(x,y_u) - \frac{\partial f}{\partial y}(x,y_{u_k})\Big]\varphi\Big\|_{L^2(\Omega)}\\
&\le C_\Omega C_{f,M}\|y_u - y_{u_k}\|_{C(\bar\Omega)}\|\varphi\|_{L^2(\Omega)} \le C_\Omega^2 C_{f,M}\|y_u - y_{u_k}\|_{C(\bar\Omega)}\stackrel{k \to \infty}{\longrightarrow} 0.
%0\ \text{ when } k \to \infty.
\end{align*}
Hence, we can proceed as in the proof of Theorem \ref{T2.5} to deduce the existence of $k_0$ such that
\[
\|[\mathcal{A}_k^*]^{-1}\|_{\mathcal{L}(H^{-1}(\Omega),H_0^1(\Omega))} \le 2 \|[\mathcal{A}^*]^{-1}\|_{\mathcal{L}(H^{-1}(\Omega),H_0^1(\Omega))} \ \ \forall k \ge k_0.
\]
In addition, arguing as in Step 2 of the proof of Theorem \ref{T2.5}, we have
\[
\|[\mathcal{A}^*]^{-1} - [\mathcal{A}_k^*]^{-1}\|_{\mathcal{L}(H^{-1}(\Omega),H_0^1(\Omega))} \to 0\ \ \text{ when } k \to \infty.
\]
Hence, we get with \eqref{E3.2}
\begin{align*}
&\|\varphi_u - \varphi_{u_k}\|_{H_0^1(\Omega)} = \Big\|[\mathcal{A}^*]^{-1}\frac{\partial L}{\partial y}(x,y_u)  - [\mathcal{A}_k^*]^{-1}\frac{\partial L}{\partial y}(x,y_{u_k})\Big\|_{H^1_0(\Omega)}\\
& \le \Big\|([\mathcal{A}^*]^{-1} - [\mathcal{A}_k^*]^{-1})\frac{\partial L}{\partial y}(x,y_u)\Big\|_{H^1_0(\Omega)}\\
&+ \Big\|[\mathcal{A}_k^*]^{-1}\Big[\frac{\partial L}{\partial y}(x,y_u)  - \frac{\partial L}{\partial y}(x,y_{u_k})\Big]\Big\|_{H^1_0(\Omega)} \\
&\le \|[\mathcal{A}^*]^{-1} - [\mathcal{A}_k^*]^{-1}\|_{\mathcal{L}(H^{-1}(\Omega),H_0^1(\Omega))}\Big\|\frac{\partial L}{\partial y}(x,y_u)\Big\|_{H^{-1}(\Omega)}\\
& + 2\|[\mathcal{A}^*]^{-1}\|_{\mathcal{L}(H^{-1}(\Omega),H_0^1(\Omega))}\Big\|\frac{\partial L}{\partial y}(x,y_u)  - \frac{\partial L}{\partial y}(x,y_{u_k})\Big\|_{H^{-1}(\Omega)}\\
&\le \|[\mathcal{A}^*]^{-1} - [\mathcal{A}_k^*]^{-1}\|_{\mathcal{L}(H^{-1}(\Omega),H_0^1(\Omega))}\Big\|\frac{\partial L}{\partial y}(x,y_u)\Big\|_{H^{-1}(\Omega)}\\
& + 2C_\Omega C_{L,M}\|[\mathcal{A}^*]^{-1}\|_{\mathcal{L}(H^{-1}(\Omega),H_0^1(\Omega))}\|y_u - y_{u_k}\|_{L^2(\Omega)} \to 0.
\end{align*}

It remains to prove that $\|\varphi_u - \varphi_{u_k}\|_{C(\bar\Omega)} \to 0$.
The equation satisfied by $\varphi_{u_k}$ can be written as
\[A^*\varphi_{u_k}-\div(b(x)\varphi_{u_k}) = -\frac{\partial f}{\partial y}(x,y_{u_k})\varphi_{u_k}+\frac{\partial L}{\partial y}(x,y_{u_k}).\]
Since $\{\varphi_{u_k}\}_{k=1}^{\infty}$ is convergent in $H^1_0(\Omega)$, we have that it is bounded in $L^2(\Omega)$. Using this, the fact that $\|y_{u_k}\|_{C(\bar\Omega)}\leq M$, and Assumptions 3 and 4, we have that the right hand side is bounded in $L^{\bar p}(\Omega)$. Hence from Corollary \ref{C2.4} we deduce that $\{\varphi_{u_k}\}_{k=1}^{\infty}$ is bounded in $C^{0,\mu}(\bar\Omega)$ for some $\mu >0$. Then, the convergence in $C(\bar\Omega)$ follows from the compact embedding $C^{0,\mu}(\bar\Omega)\subset C(\bar\Omega)$.
\end{proof}

The following corollary is an immediate consequence of Theorem \ref{T3.8} and Lemma \ref{L3.5}.
\begin{crllr}
Under the assumptions of Theorem \ref{T3.8}. there exist $\kappa > 0$ and $\varepsilon > 0$ such that
\begin{equation}
J(\bar u) + \frac{\kappa}{2}\|u - \bar u\|^2_{L^2(\Omega)} \le J(u) \ \ \forall u \in \uad : \|u - \bar u\|_{L^2(\Omega)} \leq \varepsilon.
\label{E3.13}
\end{equation}
\label{C3.9}
\end{crllr}

It is interesting to remark that if $\bar u$ satisfies \eqref{E3.7}-\eqref{E3.9} and \eqref{E3.11}, then besides of being a strict local solution of \Pb, there exists a ball $B_r(\bar u) \subset L^2(\Omega)$ such that $\bar u$ is the unique stationary point of \Pb in $B_r(\bar u)$, i.e.~the unique control satisfying \eqref{E3.7}-\eqref{E3.9}; see \cite[Corollary 2.6]{Casas-Troltzsch2012}.

\end{document}